\documentclass[a4paper,10pt,twoside]{amsart}

\usepackage{graphics}
\usepackage{amsfonts}
\usepackage{amsmath}

\numberwithin{equation}{section}
\def\RR{\mathbb{R}}
\def\MM{\mathbb{M}}

\def\eps{\varepsilon}
\def\wto{\rightharpoonup}
\def\Aff{{\rm Aff}}
\def\det{{\rm det}}
\def\Z{\mathcal{Z}}

\def\rank{{\rm rank}}
\def\dom{{\rm dom}}

\newtheorem{theorem}{Theorem}[section]
\newtheorem{lemma}[theorem]{Lemma}
\newtheorem{proposition}[theorem]{Proposition}
\newtheorem{corollary}[theorem]{Corollary}

\theoremstyle{definition}
\newtheorem{definition}[theorem]{Definition}

\theoremstyle{remark}
\newtheorem{remark}[theorem]{Remark}

\begin{document}

\title{On the homogenization of singular integrals}

\author[O. Anza Hafsa]{Omar Anza Hafsa}
\address{UNIVERSITE MONTPELLIER II, UMR-CNRS 5508, LMGC, Place Eug\`ene Bataillon, 34095 Montpellier, France.}
\email{Omar.Anza-Hafsa@univ-montp2.fr}

\author[M. L. Leghmizi]{Mohamed Lamine Leghmizi}
\address{UNIVERSITE DE MEDEA, LMP2M (Laboratoire de M\'ecanique, Physique et Mod\'elisation Math\'ematique), Quartier Ain-D'Heb, M\'ed\'ea (26000), Algerie. }
\email{leghmizi@hotmail.com}

\author[J.-P. Mandallena]{Jean-Philippe Mandallena}
\address{UNIVERSITE DE NIMES, Site des Carmes, Place Gabriel P\'eri, 30021 N\^\i mes, France.}
\email{jean-philippe.mandallena@unimes.fr}

\date{March 2009}

\begin{abstract}
We study periodic homogenization by $\Gamma$-convergence of some singular integral functionals related to nonlinear elasticity.
\end{abstract}

\maketitle


\section{Introduction}

Consider the family of  integral functionals $\{I_\eps\}_{\eps>0}$ given by
\begin{equation}\label{IntegralFunct}
I_\eps(\phi):=\int_{\Omega}W\left({x\over\eps},\nabla\phi(x)\right)dx,
\end{equation}
where $\eps>0$ is a (small) parameter, $\Omega\subset\RR^N$ a bounded open set with $|\partial\Omega|=0$ (where $|\cdot|$ denotes the Lebesgue measure in $\RR^N$), $\phi\in W^{1,p}(\Omega;\RR^m)$ with $p\geq1$ and $W:\RR^N\times \MM^{m\times N}\to[0,+\infty]$ assumed to be a normal integrand\footnote{A function $W:\RR^N\times\MM^{m\times N}\to[0,+\infty]$ is called a normal integrand if for every $x\in\RR^N$, $W(x,\cdot)$ is lower semicontinuous and $W$ is measurable.}, where $\MM^{m\times N}$ denotes the space of real $m\times N$ matrices, is $p$-coercive, i.e., 
\begin{equation}\label{CoercivityOfW}
W(x,\xi)\geq C|\xi|^p\hbox{ for all }(x,\xi)\in\RR^N\times\MM^{m\times N}\hbox{ and some }C>0,
\end{equation}
$1$-periodic, i.e.,  
\begin{equation}\label{PeriodicityOfW}
W(x+e_i,\xi)=W(x,\xi) \hbox{ for all }(x,\xi)\in\RR^N\times\MM^{m\times N}\hbox{ and }i=1,\cdots,N,
\end{equation}
where $(e_1,\cdots,e_N)$ is the standard basis of $\RR^N$, and of $p$-polynomial growth, i.e.,
\begin{equation}\label{GrowthCondition}
W(x,\xi)\leq c(1+|\xi|^p) \hbox{ for all }(x,\xi)\in\RR^N\times\MM^{m\times N}\hbox{ and some }c>0.
\end{equation}
In \cite{braides85} (see also \cite[Theorem 14.5 p. 111]{braides-defranceschi98}) Braides proved that $I_\eps$ $\Gamma$-converge with respect to the $L^p(\Omega;\RR^m)$-convergence as $\eps\to 0$ (see Definition \ref{DefGamma-Convergence}) to the functional $I_{\rm hom}$ defined on $W^{1,p}(\Omega;\RR^m)$ by
\begin{equation}\label{IntegralFunctHom}
I_{\rm hom}(\phi):=\int_\Omega W_{\rm hom}(\nabla\phi(x))dx
\end{equation}
with $W_{\rm hom}:\MM^{m\times N}\to[0,+\infty]$ given by
\begin{equation}\label{Homogenized-Density}
W_{\rm hom}(\xi):=\inf_{k\geq 1}{1\over k^N}\inf\left\{\int_{kY}W(x,\xi+\nabla\varphi(x))dx:\varphi\in W^{1,p}_0(kY;\RR^m)\right\}
\end{equation}
with $Y:=]0,1[^N$ and $W^{1,p}_0(kY;\RR^m):=\{\varphi\in W^{1,p}(kY;\RR^m):\varphi=0\hbox{ on }\partial(kY)\}$. This result established a suitable variational framework to deal with homogenization problems in the vectorial case: it is the point of departure of many works on the subject related to nonlinear elasticity. However, because of the $p$-polynomial growth assumption \eqref{GrowthCondition}, Braides's homogenization theorem is not compatible with the following two important physical properties: the noninterpenetration of the matter, i.e., $W(x,\xi)=+\infty$ if and only if $\det \xi\leq 0$, and the necessity of an infinite amount of energy to compress a finite volume into zero volume, i.e., $W(x,\xi)\to +\infty$ as $\det \xi\to 0$, where $\det \xi$ denotes the determinant of the $N\times N$ matrix $\xi$. 

In this paper we show that by using Braides's homogenization theorem (see Theorem \ref{BraidesTheorem}) and a slight generalization of a relaxation theorem (see Theorem \ref{GeneralizedRelaxationTheorem}), that we obtained in \cite{oah-jpm07,oah-jpm08a}, it is possible to establish a homogenization theorem (see Theorem \ref{HomogenizationTheorem}) which applies to functionals of type \eqref{IntegralFunct} when the integrand is singular (see Corollary \ref{corollary1}). A typical example of a such singular integrand is given by $W:\RR^N\times\MM^{N\times N}\to[0,+\infty]$ of the form
\begin{equation}\label{TypicalIntegrand}
W(x,\xi)=
 |\xi|^p+a(x)h(\det \xi)
\end{equation}
where $0<\eta\leq a\in L^\infty(\RR^N)\cap C(\RR^N)$ is a $1$-periodic function  and $h:\RR\to[0,+\infty]$ is a measurable function for which there exist $\gamma,\delta>0$ such that $h(t)\leq\delta$ for all $|t|\geq\gamma$. For example, given $s>0$ and $T\geq0$ (possibly very large), this latter condition is satisfied with $\gamma=2T$ and $\delta=\max\{{1\over (2T)^s},T\}$ when $h$ is of type  
\begin{equation}\label{TypeOfH}
h(t)=\left\{
\begin{array}{cl}
T&\hbox{if }t<-T\\
+\infty&\hbox{if }t\in[-T,0]\\
\displaystyle{1\over t^s}&\hbox{if }t>0.
\end{array}
\right.
\end{equation}
Note that $W$ as in \eqref{TypicalIntegrand} with $h$ given by \eqref{TypeOfH} is compatible with the singular behavior  $W(x,\xi)\to+\infty$ as $\det \xi\to 0$ (however, such a $W$ is not consistent with the noninterpenetration of the matter). 

An outline of the paper is as follows. Our homogenization theorem (see Theorem \ref{HomogenizationTheorem}) is stated and proved in \S 3. Its proof uses a relaxation theorem (see Theorem \ref{GeneralizedRelaxationTheorem}), whose statement and proof are given in \S 2, and Braides's homogenization theorem (see Theorem \ref{BraidesTheorem}).  Homogenization of functionals of type \eqref{IntegralFunct} when $W$ is of the form \eqref{TypicalIntegrand}-\eqref{TypeOfH} is treated in \S 4 as an application of Theorem \ref{HomogenizationTheorem} (see Corollary \ref{corollary1}).


\section{Relaxation theorem}

Let $m,N\geq 1$ be two integers. Given any bounded open set $D\subset\RR^N$ with $|\partial D|=0$, we denote  the space of continuous piecewise affine functions from $D$ to $\RR^m$ by $\Aff(D;\RR^m)$, i.e., $\varphi\in\Aff(D;\RR^m)$ if and only if $\varphi$ is continuous and there exists a finite family $\{D_i\}_{i\in I}$ of open disjoint subsets of $D$ such that $|\partial D_i|=0$ for all $i\in I$, $|D\setminus \cup_{i\in I} D_i|=0$ and for every $i\in I$, $\nabla \varphi\equiv \xi_i$ in $D_i$ with $\xi_i\in\MM^{m\times N}$, and we set $\Aff_0(D;\RR^m):=\{\varphi\in\Aff(D;\RR^m):\varphi=0\hbox{ on }\partial D\}$ and $W^{1,\infty}_0(D;\RR^m):=\{\varphi\in W^{1,\infty}(D;\RR^m):\varphi=0\hbox{ on }\partial D\}$. Given a normal integrand $f:\RR^N\times\MM^{m\times N}\to[0,+\infty]$, where $\MM^{m\times N}$ denotes the space of real $m\times N$ matrices, we consider the normal integrand $\Z f:\RR^N\times\MM^{m\times N}\to[0,+\infty]$ defined by
\[
\Z f(x,\xi):=\inf\left\{\int_Y f(x,\xi+\nabla\varphi(y))dy:\varphi\in W^{1,\infty}_0(Y;\RR^m)\right\}
\]
with $Y:=]0,1[^N$. The following result is due to Fonseca (see \cite[lemma 2.16, Theorem 2.17 and Proposition 2.3]{fonseca88}).
\begin{proposition}\label{FonsecaProperties}
The function $\Z f$ satisfies the following properties.
\begin{itemize}
\item[{\rm(a)}] For every bounded open set $D\subset\RR^N$ with $|\partial D|=0$ and every $(x,\xi)\in\RR^N\times\MM^{m\times N}$,
\[
\Z f(x,\xi)=\inf\left\{{1\over |D|}\int_D f(x,\xi+\nabla\varphi(y))dy:\varphi\in W^{1,\infty}_0(D;\RR^m)\right\}.
\]
\item[{\rm(b)}] For every $x\in\RR^N$, if $\Z f(x,\cdot)$ is finite then $\Z f(x,\cdot)$ is rank-one convex, i.e.,  for every $\xi,\xi^\prime\in\MM^{m\times N}$ with $\rank(\xi-\xi^\prime)\leq 1$,
\[
\Z f(x,\lambda\xi+(1-\lambda)\xi^\prime)\leq\lambda\Z f(x,\xi)+(1-\lambda)\Z f(x,\xi^\prime).
\] 
\item[{\rm(c)}] For every $x\in\RR^N$, if $\Z f(x,\cdot)$ is finite then $\Z f(x,\cdot)$ is  continuous, i.e., $\Z f$ is a Carath\'eodory integrand\footnote{A function $f:\RR^N\times\MM^{m\times N}\to[0,+\infty]$ is called a Carath\'eodory integrand if $f(x,\xi)$ is measurable in $x$ and continuous in $\xi$.} whenever $\Z f$ is finite.
\item[{\rm(d)}] For every bounded open set $D\subset\RR^N$ with $|\partial D|=0$, every $(x,\xi)\in\RR^N\times\MM^{m\times N}$ and every $\varphi\in\Aff_0(D;\RR^m)$,
$$
\Z f(x,\xi)\leq{1\over |D|}\int_D \Z f(x,\xi+\nabla\varphi(y))dy.
$$
\end{itemize}
\end{proposition}

\begin{remark}\label{FonsecaPropertiesRemark}
Proposition \ref{FonsecaProperties} is also valid with ``$\hat \Z f$" instead of ``$\Z f$" (see \cite[Proposition 2.3]{oah-jpm09}) where  $\hat\Z f:\RR^N\times\MM^{m\times N}\to[0,+\infty]$ is given by
\[
\hat\Z f(x,\xi):=\inf\left\{\int_Y f(x,\xi+\nabla\varphi(y))dy:\varphi\in\Aff_0(Y;\RR^m)\right\}.
\]
In particular, Proposition \ref{FonsecaProperties}(d)  can be rewritten as $\hat\Z[\Z f]=\Z f$. \end{remark}

Given $x\in\RR^N$ we say that $f(x,\cdot)$ is quasiconvex (in the sense of Morrey \cite{morrey52}) if for every $\xi\in\MM^{m\times N}$, every bounded open set $D\subset\RR^N$ with $|\partial D|=0$ and every $\varphi\in W^{1,\infty}_0(D;\RR^m)$,
$$
f(x,\xi)\leq{1\over|D|}\int_D f(x,\xi+\nabla\varphi(y))dy.
$$ 
By the quasiconvex envelope of $f(x,\cdot)$, that we denote by $\mathcal{Q}f(x,\cdot)$, we mean the greatest quasiconvex function which less than or equal to $f(x,\cdot)$. (clearly, $f(x,\cdot)$ is quasiconvex if and only if $\mathcal{Q}f(x,\cdot)=f(x,\cdot)$.) The concept of quasiconvex envelope was introduced by Dacorogna (see \cite{dacorogna82}) who proved the following theorem (see \cite[Theorem 6.9 p. 271]{dacorogna08}).
\begin{theorem}\label{DacorognaTheorem}
If $f$ is finite then $\mathcal{Q}f=\hat \Z f=\Z f$.
\end{theorem}

The following result  is a slight generalization of Theorem \ref{DacorognaTheorem}.

\newtheorem*{DacorognaTheorem}{\bf Theorem \ref{DacorognaTheorem}-bis}
\begin{DacorognaTheorem}
{\em If $\Z f$ is finite then $\mathcal{Q}f=\Z f$. In particular, $\Z f(x,\cdot)$ is quasiconvex for all $x\in\RR^N$.}
\end{DacorognaTheorem}
\begin{proof}
As $\Z f$ is finite we have $\mathcal{Q}[\Z f]=\hat\Z[\Z f]$ by Theorem \ref{DacorognaTheorem}. But $\hat\Z[\Z f]=\Z f$ from Remark \ref{FonsecaPropertiesRemark} and so $\mathcal{Q}[\Z f]=\Z f$, i.e., $\Z f(x,\cdot)$ is quasiconvex for all $x\in\RR^N$. As $\Z f\geq\mathcal{Q} f$ it follows that $\mathcal{Q}f=\Z f$.
 \end{proof}

\begin{remark}
Theorem \ref{DacorognaTheorem} can be also generalized as follows: if $\hat\Z f$ is finite then $\mathcal{Q} f=\hat\Z f=\Z f$ (see \cite[Corollaire 2.17]{oah-jpm09}).
\end{remark}

From now on we fix $p\geq 1$. Given $U\subset\RR^N$ be a bounded open set with $|\partial U|=0$ we define $F:W^{1,p}(U;\RR^m)\to[0,+\infty]$ by
\[
F(\phi):=\int_U f(x,\nabla\phi(x))dx
\]
and we consider the relaxed functionals $\overline{F},\overline{F}_0:W^{1,p}(U;\RR^m)\to[0,+\infty]$ given by: 
\begin{itemize}
\item[$\diamond$] $\overline{F}(\phi):=\inf\left\{\liminf\limits_{n\to+\infty}F(\phi_n):\phi_n\to\phi\hbox{ in }L^p(U;\RR^m)\right\}$;
\item[$\diamond$] $\overline{F}_0(\phi):=\inf\left\{\liminf\limits_{n\to+\infty}F(\phi_n):W^{1,p}_0(U;\RR^m)\ni\phi_n\to\phi\hbox{ in }L^p(U;\RR^m)\right\}$
\end{itemize}
with $W^{1,p}_0(U;\RR^m):=\{\phi\in W^{1,p}(U;\RR^m):\phi=0\hbox{ on }\partial U\}$. As $\overline{F}$ and $\overline{F}_0$ are not given by explicit formulas, it is of interest to know under which conditions on $f$ we have:
\begin{eqnarray}
&&\displaystyle\overline{F}(\phi)=\int_U\overline{f}(x,\nabla\phi(x))dx\hbox{ for all }\phi\in W^{1,p}(U;\RR^m);\label{IntRep1}\\
&&\overline{F}_0(\phi)=\left\{
\begin{array}{cl}
\overline{F}(\phi)&\hbox{if }\phi\in W^{1,p}_0(U;\RR^m)\\
+\infty&\hbox{otherwise}
\end{array}
\right.\label{IntRep2}
\end{eqnarray}
with $\overline{f}:\RR^N\times\MM^{m\times N}\to[0,+\infty]$ (whose we wish to give a representation formula). In the $p$-polynomial growth case, such integral representation problems was studied by Dacorogna (see \cite[Theorem 5]{dacorogna82}, see also \cite[Theorem 9.1 p. 416]{dacorogna08}) and Acerbi and Fusco (see \cite[Statement III.7]{acerbi-fusco84}) who proved the following theorem.
\begin{theorem}\label{Acerbi-Dacorogna-FuscoTheorem}
Let $f:\RR^N\times\MM^{m\times N}\to[0,+\infty]$ be a $p$-coercive Carath\'eodory integrand. If $f$ is of $p$-polynomial growth, i.e., 
\begin{equation}\label{fGrowthCondition}
f(x,\xi)\leq c(1+|\xi|^p)\hbox{ for all }(x,\xi)\in\RR^N\times\MM^{m\times N}\hbox{ and some }c>0,
\end{equation}
then \eqref{IntRep1} and \eqref{IntRep2} hold with $\overline{f}=\Z f=\mathcal{Q} f$. If moreover $f(x,.)$ is quasiconvex for all $x\in\RR^N$ then $\overline{f}=f$.
\end{theorem}

Because of the assumption \eqref{fGrowthCondition} Acerbi-Dacorogna-Fusco's relaxation theorem cannot handle integrands having a singular behavior of type $f(x,\xi)\to+\infty$ as $\det\xi\to 0$ (when $m=N$). However, by using Theorem \ref{Acerbi-Dacorogna-FuscoTheorem} and a key lemma (see Lemma \ref{KeyLemma}) we can go beyond the $p$-polynomial growth case (see Theorem \ref{GeneralizedRelaxationTheorem}  and Corollary \ref{corollary2}).

Let $\mathcal{A}^p$ be the class of $p$-coercive normal integrands $f:\RR^N\times\MM^{m\times N}\to[0,+\infty]$ satisfying the following two conditions:
\begin{itemize}
\item[(C$_1$)] there exists a function $\omega:[0,+\infty[\to[0,+\infty[$ continuous at the origin with $\omega(0)=0$ such that for every $x_1,x_2\in\RR^N$ and every $\xi\in \MM^{m\times N}$,
\[
f(x_1,\xi)\leq \omega(|x_1-x_2|)(1+f(x_2,\xi))+f(x_2,\xi);
\]
\item[(C$_2$)] the function $\Z f$ is of $p$-polynomial growth, i.e., $\Z f(x,\xi)\leq c(1+|\xi|^p)$ for all $(x,\xi)\in\RR^N\times\MM^{m\times N}$ and some $c>0$.
\end{itemize}
\begin{remark}\label{ElementaryRemark}
\begin{itemize}
\item[(i)] Condition (C$_1$) is a condition of Serrin type (see \cite{serrin61} or \cite[p. 96-97]{morrey66}, see also \cite{marcellini86}).

\item[(ii)] If $f$ satisfies (C$_1$) then ${\dom}f(x_1,\cdot)={\dom}f(x_2,\cdot)$ for all $x_1,x_2\in\RR^N$, where, for $x\in\RR^N$, ${\dom}f(x,\cdot):=\{\xi\in\MM^{m\times N}:f(x,\xi)<+\infty\}$.

\item[(iii)] If $f$ satisfies (C$_1$) then $f(\cdot,\xi)$ is continuous for all $\xi\in\MM^{m\times N}$.

\item[(iv)] If $f$ satisfies {\rm(C$_1$)} then  for every $x_1,x_2\in\RR^N$ and every $\xi\in \MM^{m\times N}$, 
\[
\Z f(x_1,\xi)\leq \omega(|x_1-x_2|)(1+\Z f(x_2,\xi))+\Z f(x_2,\xi).
\]
In particular, we have  $\Z f(\cdot,\xi)\in C(\overline{V})$ for all $\xi\in\MM^{m\times N}$ whenever $f$ satisfies (C$_1$), $\Z f$ is finite (for example when (C$_2$) holds) and $V\subset\RR^N$ is a bounded open set. (This fact will be used in the proof of Lemma \ref{KeyLemma}.) 
\end{itemize}
\end{remark}

The following theorem is a slight generalization of \cite[Theorem 2]{oah-jpm07} (see also \cite[Theorem 1.4]{oah-jpm08a} and \cite[Th\'eor\`eme 2.1]{oah-jpm09}).
\begin{theorem}\label{GeneralizedRelaxationTheorem}
If $f\in\mathcal{A}^p$ then \eqref{IntRep1} and \eqref{IntRep2} hold with $\overline{f}=\Z f=\mathcal{Q} f$. As a consequence, we have
\[
\inf_{\phi\in W^{1,p}_0(U;\RR^m)}\int_U f(x,\nabla\phi(x))dx=\inf_{\phi\in W^{1,p}_0(U;\RR^m)}\int_U \Z f(x,\nabla\phi(x))dx.
\]
\end{theorem}
\begin{proof}
Let $\Z F:W^{1,p}(U;\RR^m)\to[0,+\infty]$ be defined by
\[
\Z F(\phi):=\int_U\Z f(x,\nabla\phi(x))dx
\]
and let $\overline{\Z F},\overline{\Z F}_0:W^{1,p}(U;\RR^m)\to[0,+\infty]$ be given by:
\begin{itemize}
\item[$\diamond$] $\overline{\Z F}(\phi):=\inf\left\{\liminf\limits_{n\to+\infty}\Z F(\phi_n):\phi_n\to\phi\hbox{ in }L^p(U;\RR^m)\right\};$
\item[$\diamond$] $\overline{\Z F}_0(\phi):=\inf\left\{\liminf\limits_{n\to+\infty}\Z F(\phi_n):W^{1,p}_0(U;\RR^m)\ni\phi_n\to\phi\hbox{ in }L^p(U;\RR^m)\right\}.$
\end{itemize}
We need the following lemma whose proof is given below.
\begin{lemma}\label{KeyLemma}
Under {\rm(C$_1$)} and {\rm(C$_2$)} if $\phi\in\Aff(U;\RR^m)$ {\rm(}resp. $\phi\in\Aff_0(U;\RR^m)${\rm)} then
\begin{equation}\label{CoNCLuSIONLeMMaKeY}
\overline{F}(\phi)\leq\int_U \Z f\left(x,\nabla\phi(x)\right)dx\ \hbox{{\rm (}resp. }\overline{F}_0(\phi)\leq\int_U \Z f\left(x,\nabla\phi(x)\right)dx\hbox{\rm)}.
\end{equation}
\end{lemma}
As $\Z f$ is of $p$-polynomial growth and  $\Aff(U;\RR^m)$ (resp. $\Aff_0(U;\RR^m)$) is strongly dense in $W^{1,p}(U;\RR^m)$ (resp. $W^{1,p}_0(U;\RR^m)$), from Lemma \ref{KeyLemma} we deduce that \eqref{CoNCLuSIONLeMMaKeY} holds for all  $\phi\in W^{1,p}(U;\RR^m)$ {\rm(}resp. $\phi\in W^{1,p}_0(U;\RR^m)${\rm)}. Thus $\overline{F}\leq\overline{\Z F}$ (resp. $\overline{F}_0\leq\overline{\Z F}_0$). Moreover,  $\overline{\Z F}\leq\overline{F}$ (resp. $\overline{\Z F}_0\leq\overline{F}_0$), hence
\begin{equation}\label{Eq1ProofRelaxTheorem}
\overline{F}=\overline{\Z F}\hbox{ (resp. }\overline{F}_0=\overline{\Z F}_0\hbox{)}.
\end{equation}
As $f$ is $p$-coercive, also is $\Z f$. Moreover, since $\Z f$ is finite (because (C$_2$) holds), on the one hand, $\Z f$ is a Carath\'eodory integrand by Proposition \ref{FonsecaProperties}(c) and, on the other hand, $\Z f(x,\cdot)$ is quasiconvex for all $x\in\RR^N$ by Theorem \ref{DacorognaTheorem}-bis. From Acerbi-Dacorogna-Fusco's relaxation theorem (see Theorem \ref{Acerbi-Dacorogna-FuscoTheorem}) it follows that 
\[
\overline{\Z F}=\Z F\hbox{ (resp. }\overline{\Z F}_0=\left\{
\begin{array}{cl}
\Z F &\hbox{on }W^{1,p}_0(U;\RR^m)\\
+\infty&\hbox{elsewhere}
\end{array}
\right.
\hbox{)}
\]
which gives the theorem when combined with \eqref{Eq1ProofRelaxTheorem}.
\end{proof}

\begin{proof}[Proof of Lemma {\rm\ref{KeyLemma}}]
By definition, there exists a finite family $\{U_i\}_{i\in I}$ of open disjoint subsets of $U$ such that $|\partial U_i|=0$ for all $i\in I$, $|U\setminus\cup_{i\in I}U_i|=0$ and, for every $i\in I$, $\nabla\phi\equiv\xi_i$ in $U_i$ with $\xi_i\in\MM^{m\times N}$. Thus
\[
\int_U\Z f(x,\nabla\phi(x))dx=\sum_{i\in I}\int_{U_i}\Z f(x,\xi_i)dx.
\]
From Remark \ref{ElementaryRemark}(iv) we see that $\Z f(\cdot,\xi_i)\in C(\overline{U}_i)$ for all $i\in I$. Hence, for each $i\in I$, there exists a finite family $\{U_{i,j}^k\}_{j\in J_{i}^k}$ of disjoint subsets of $U_i$ with $|\partial U^k_{i,j}|=0$ for all $j\in J_i^k$ and $|U_i\setminus\cup_{j\in J_{i}^k}U_{i,j}^k|=0$ such that:
\begin{eqnarray}
&&{\rm diam}(U^k_{i,j})<{1\over k}\hbox{ for all }j\in J^k_i;\label{DiamHypo}\\
&&\lim_{k\to+\infty}\sum_{j\in J_{i}^k}|U_{i,j}^k|\Z f(x_{i,j}^k,\xi_i)=\int_{U_i}\Z f(x,\xi_i)dx,\label{RiemannSum}
\end{eqnarray}
where, for $X\subset\RR^N$, ${\rm diam}(X):=\sup\{|x_1-x_2|:x_1,x_2\in X\}$. Fix any $\delta>0$. Then, there exists $\eta>0$ such that  
\begin{equation}\label{ContiNuiTyoFomega}
\omega(t)<{\delta}\hbox{ for all }|t|<{\eta},
\end{equation}
where the function $\omega:[0,+\infty[\to[0,+\infty[$ (continuous at the origin with $\omega(0)=0$) are given by (C$_1$).  Fix any $k\geq 1$ such that ${1\over k}<\eta$. Given any $i\in I$ and any $j\in J_i^k$ we consider $\varphi_{i,j}\in W^{1,\infty}_0(Y;\RR^m)$ such that
\begin{equation}\label{ZfAssumptioN}
\int_Y f(x_{i,j}^k,\xi_i+\nabla\varphi_{i,j}(y))dy\leq\Z f(x_{i,j}^k,\xi_i)+{\delta\over|U|}.
\end{equation}
By Vitali's covering theorem, there exists a finite or countable family $\{a_{i,j,\ell}+\alpha_{i,j,\ell} Y\}_{\ell\in L_{i,j}}$ of disjoint subsets of $U_{i,j}^k$, where $a_{i,j,\ell}\in\RR^N$ and $0<\alpha_{i,j,\ell}<{1\over n}$, such that $|U_{i,j}^k\setminus\cup_{\ell\in L_{i,j}}(a_{i,j,\ell}+\alpha_{i,j,\ell} Y)|=0$ (and so $\sum_{\ell\in L_{i,j}}\alpha_{i,j,\ell}^N=|U_{i,j}^k|$). Define $\phi_n\in W^{1,p}_0(U;\RR^m)$ by
\[
\phi_{n}(x):=\alpha_{i,j,\ell}\varphi_{i,j}\left({x-a_{i,j,\ell}\over \alpha_{i,j,\ell}}\right)\hbox{ if }x\in a_{i,j,\ell}+\alpha_{i,j,\ell}Y.
\]
Then:
\begin{itemize}
\item[$\diamond$] $\|\phi_n\|_{L ^\infty(U;\RR^m)}\leq{1\over n}\max_{i\in I, j\in J_i^k}\|\varphi_{i,j}\|_{L^\infty(Y;\RR^m)}$ for all $n\geq 1$;
\item[$\diamond$] $\|\nabla\phi_n\|_{L^\infty(U;\RR^m)}\leq\max_{i\in I, j\in J_i^k}\|\nabla\varphi_{i,j}\|_{L^\infty(Y;\RR^m)}$ for all $n\geq 1$,
\end{itemize}
hence (up to a subsequence) $\phi_n\stackrel{*}{\wto}0$ in $W^{1,\infty}(U;\RR^m)$, where ``$\stackrel{*}{\wto}$" denotes the weak$^*$ convergence in $W^{1,\infty}(U;\RR^m)$. Consequently $\phi_n\wto 0$ in $W^{1,p}(U;\RR^m)$ and so  (up to a subsequence) $\phi_n\to 0$ in $L^p(U;\RR^m)$. Using \eqref{ZfAssumptioN} and (C$_2$) we see that 
\[
\sum_{i\in I}\sum_{j\in J_i^k}\int_{U_{i,j}^k}f(x_{i,j}^k,\xi_i+\nabla\phi_n(x))dx\leq {c\over C}(|U|+\|\nabla\phi\|_{L^p(U;\RR^m)}^p)+{\delta\over C}\hbox{ for all }n\geq 1
\]
with $c,C>0$. Taking (C$_1$), \eqref{DiamHypo} and \eqref{ContiNuiTyoFomega} into account it follows that for every $n\geq 1$,
\[
\int_U f\left(x,\nabla\phi(x)+\nabla\phi_n(x)\right)dx\leq K(\delta)+\sum_{i\in I}\sum_{j\in J_i^k}\int_{U_{i,j}^k}f(x_{i,j}^k,\xi_i+\nabla\phi_n(x))dx
\]
with $\delta(|U|+{c\over C}(|U|+\|\nabla\phi\|_{L^p(U;\RR^m)}^p)+{\delta\over C})=:K(\delta)\to 0$ as $\delta\to 0$. Moreover, from \eqref{ZfAssumptioN} we have
\begin{eqnarray*}
\sum_{i\in I}\sum_{j\in J_i^k}\int_{U_{i,j}^k}f(x_{i,j}^k,\xi_i+\nabla\phi_n(x))dx&=&\sum_{i\in I}\sum_{j\in J_i^k}|U^k_{i,j}|\int_Yf(x_{i,j}^k,\xi_i+\nabla\varphi_{i,j}(x))dx\\
&\leq&\sum_{i\in I}\sum_{j\in J_{i}^k}|U_{i,j}^k|\Z f(x_{i,j}^k,\xi_i)+\delta,
\end{eqnarray*}
hence, for every $n\geq 1$,
\[
\int_U f\left(x,\nabla\phi(x)+\nabla\phi_n(x)\right)dx\leq \sum_{i\in I}\sum_{j\in J_{i}^k}|U_{i,j}^k|\Z f(x_{i,j}^k,\xi_i)+K(\delta)+\delta.
\]
As $\phi+\phi_n\in W^{1,p}(U;\RR^m)$ (resp. $\phi+\phi_n\in W^{1,p}_0(U;\RR^m)$) for all $n\geq 1$ and $\phi+\phi_n\to\phi$ in  $L^p(U;\RR^m)$ it follows that 
\[
\overline{F}(\phi)\leq \sum_{i\in I}\sum_{j\in J_{i}^k}|U_{i,j}^k|\Z f(x_{i,j}^k,\xi_i)+K(\delta)+\delta
\]
\[
\hbox{(resp. }\overline{F}_0(\phi)\leq \sum_{i\in I}\sum_{j\in J_{i}^k}|U_{i,j}^k|\Z f(x_{i,j}^k,\xi_i)+K(\delta)+\delta\hbox{).}
\]
Letting $k\to+\infty$ and using \eqref{RiemannSum} we deduce that
\[
\overline{F}(\phi)\leq\int_U\Z f(x,\nabla\phi(x))dx+K(\delta)+\delta
\]
\[
\hbox{ (resp. }\overline{F}_0(\phi)\leq\int_U\Z f(x,\nabla\phi(x))dx+K(\delta)+\delta\hbox{)}
\]
and \eqref{CoNCLuSIONLeMMaKeY} follows by letting $\delta\to 0$.
\end{proof}


\section{Homogenization theorem} 

Let $\Omega\subset\RR^N$ be a bounded open set with $|\partial\Omega|=0$, let $W:\RR^N\times \MM^{m\times N}\to[0,+\infty]$ be a normal integrand and, for each $\eps>0$, let $I_\eps:W^{1,p}(\Omega;\RR^m)\to[0,+\infty]$ (with $p\geq 1$) be defined by \eqref{IntegralFunct}. To accomplish our asymptotic analysis as $\eps\to 0$, we will use De Giorgi's $\Gamma$-convergence which can be defined as follows (for more details see \cite{dalmaso93,braides-defranceschi98,braides06}).
\begin{definition}\label{DefGamma-Convergence}
We say that $I_\eps$ $\Gamma$-converges to $I_{\rm hom}:W^{1,p}(\Omega;\RR^m)\to[0,+\infty]$ with respect to the $L^p(\Omega;\RR^m)$-convergence as $\eps\to 0$, and we write $I_{\rm hom}=\Gamma\hbox{\rm -}\lim_{\eps\to 0}I_\eps$, if 
\[
\left(\Gamma\hbox{-}\liminf_{\eps\to 0}I_\eps\right)(\phi)=\left(\Gamma\hbox{-}\limsup_{\eps\to 0}I_\eps\right)(\phi)=I_{\rm hom}(\phi)
\]
for all $\phi\in W^{1,p}(\Omega;\RR^m)$ with:
\[
\left(\Gamma\hbox{-}\liminf_{\eps\to 0}I_\eps\right)(\phi):=\inf\left\{\liminf_{\eps\to 0}I_\eps(\phi_\eps):\phi_\eps\to \phi\hbox{ in }L^{p}(\Omega;\RR^m)\right\};
\]
\[
\left(\Gamma\hbox{-}\limsup_{\eps\to 0}I_\eps\right)(\phi):=\inf\left\{\limsup_{\eps\to 0}I_\eps(\phi_\eps):\phi_\eps\to \phi\hbox{ in }L^{p}(\Omega;\RR^m)\right\}.
\]
\end{definition}

Braides proved, in the $p$-polynomial growth case, the following theorem (see \cite{braides85,braides-defranceschi98}, see also \cite{muller87}).
\begin{theorem}\label{BraidesTheorem}
If $W$ satisfies \eqref{CoercivityOfW}, \eqref{PeriodicityOfW} and \eqref{GrowthCondition} then $I_{\rm hom}=\Gamma\hbox{\rm -}\lim_{\eps\to 0}I_\eps$ with $I_{\rm hom}$ defined by \eqref{IntegralFunctHom} and $W_{\rm hom}:\MM^{m\times N}\to[0,+\infty]$ given by \eqref{Homogenized-Density}.
\end{theorem}

To prove our homogenization theorem (see Theorem \ref{HomogenizationTheorem}) we will need Theorems \ref{GeneralizedRelaxationTheorem} and \ref{BraidesTheorem} and the following classical property of the $\Gamma$-convergence (see \cite[Proposition 2.5]{braides06}). 

\begin{proposition}\label{FundamentalProperty}
The $\Gamma$-limit is stable by substituting $I_\eps$ by its relaxed functional $\overline{I}_\eps$, i.e., 
\[
\Gamma\hbox{-}\liminf\limits_{\eps\to 0}I_\eps=\Gamma\hbox{-}\liminf\limits_{\eps\to 0}\overline{I}_\eps\hbox{ and }\Gamma\hbox{-}\limsup\limits_{\eps\to 0}I_\eps=\Gamma\hbox{-}\limsup\limits_{\eps\to 0}\overline{I}_\eps,
\]
where, for each $\eps>0$, $\overline{I}_\eps:W^{1,p}(\Omega;\RR^m)\to[0,+\infty]$ is given by
\[
\overline{I}_\eps(\phi):=\inf\left\{\liminf_{n\to+\infty}I_\eps(\phi_n):\phi_n\to\phi\hbox{ in }L^p(\Omega;\RR^m)\right\}.
\]
\end{proposition}

Set $\mathcal{A}_{\rm per}^p:=\{f\in\mathcal{A}^p:f\hbox{ is $1$-periodic}\}$. The main result of the paper is the following. (When $m=N$, it can handle integrands having a singular behavior of type $W(x,\xi)\to+\infty$ as $\det\xi\to 0$, see Corollary \ref{corollary1}).  

\begin{theorem}\label{HomogenizationTheorem}
If $W\in\mathcal{A}^p_{\rm per}$ then $I_{\rm hom}=\Gamma\hbox{\rm -}\lim_{\eps\to 0}I_\eps$ with $I_{\rm hom}$ defined by \eqref{IntegralFunctHom} and $W_{\rm hom}:\MM^{m\times N}\to[0,+\infty]$ given by \eqref{Homogenized-Density}.
\end{theorem}
\begin{proof}
By Proposition \ref{FundamentalProperty} it suffices to prove Theorem \ref{HomogenizationTheorem} with ``$\overline{I}_\eps$" instead of ``$I_\eps$". Fix any $\eps>0$ and consider $f_\eps:\RR^N\times\MM^{m\times N}\to [0,+\infty]$ given by $f_\eps(x,\xi):=W({x\over\eps},\xi)$. As $W\in\mathcal{A}^p_{\rm per}$ and $\Z f_\eps(x,\xi)=\Z W({x\over\eps},\xi)$ for all $(x,\xi)\in\RR^N\times\MM^{m\times N}$ it is easy to see that $f_\eps\in\mathcal{A}^p$. Applying Theorem \ref{GeneralizedRelaxationTheorem} (with $f=f_\eps$) we deduce that for every $\eps>0$,
\[
\overline{I}_\eps(\phi)=\int_\Omega\Z W\left({x\over\eps},\nabla\phi(x)\right)dx,
\] 
where $\Z W$ is clearly $p$-coercive, $1$-periodic and of $p$-polynomial growth. From Brai-des's homogenization theorem (see Theorem \ref{BraidesTheorem}) it follows that $I_{\rm hom}=\Gamma\hbox{\rm -}\lim_{\eps\to 0}\overline{I}_\eps$ with $I_{\rm hom}$ defined by \eqref{IntegralFunctHom} and $W_{\rm hom}:\MM^{m\times N}\to[0,+\infty]$ given by 
\[
W_{\rm hom}(\xi)=\inf_{k\geq 1}{1\over k^N}\inf\left\{\int_{kY}\Z W(x,\xi+\nabla\varphi(x))dx:\varphi\in W^{1,p}_0(kY;\RR^m)\right\}.
\]
Fix any $k\geq 1$, any $\xi\in\MM^{m\times N}$ and consider $f_\xi:\RR^N\times \MM^{m\times N}\to[0,+\infty]$ given by $f_\xi(x,\zeta):=W(x,\xi+\zeta)$. As $W\in\mathcal{A}^p_{\rm per}$ and $\Z f_\xi(x,\zeta)=\Z W(x,\xi+\zeta)$ for all $(x,\zeta)\in\RR^N\times\MM^{m\times N}$ it is easy to see that $f_\xi\in\mathcal{A}^p$. Applying Theorem \ref{GeneralizedRelaxationTheorem} (with $U=k Y$ and $f=f_\xi$) we deduce that for every $k\geq 1$ and every $\xi\in\MM^{m\times N}$,
\[
\inf_{\varphi\in W^{1,p}_0(kY;\RR^m)}\int_{kY}W(x,\xi+\nabla\varphi(x))dx=\inf_{\varphi\in W^{1,p}_0(kY;\RR^m)}\int_{kY}\Z W(x,\xi+\nabla\varphi(x))dx
\]
and the theorem follows.
\end{proof}


\section{Application}

The following condition on the normal integrand $f:\RR^N\times\MM^{N\times N}\to[0,+\infty]$ is compatible with a singular behavior of type $f(x,\xi)\to+\infty$ as $\det\xi\to 0$.
\begin{itemize}
\item[($\hat{\rm C}_2$)] There exist $\alpha,\beta>0$ such that for every $(x,\xi)\in\RR^N\times\MM^{N\times N}$,
\[
\hbox{if }|\det\xi|\geq\alpha\hbox{ then }f(x,\xi)\leq\beta(1+|\xi|^p).
\]
\end{itemize}

Typically, the function $H:\RR^N\times\MM^{N\times N}\to[0,+\infty]$ defined by
\[
H(x,\xi):=|\xi|^p+a(x)h(\det\xi),
\]
where $0\leq a\in L^\infty(\RR^N)$ and $h:\RR\to[0,+\infty[$ is a measurable function for which there exist $\gamma,\delta>0$ such that $h(t)\leq\delta$ for all $|t|\geq\gamma$, satisfies ($\hat{\rm C}_2$) with $\alpha=\gamma$ and $\beta=\max\{1,\delta\|a\|_{L^\infty(\RR^N)}\}$. The singular behavior $H(x,\xi)\to+\infty$ as $\det\xi\to 0$ is possible (for example when $h$ is given by \eqref{TypeOfH}).

Denote the class of $p$-coercive normal integrands $f:\RR^N\times\MM^{N\times N}\to [0,+\infty]$ satisfying (C$_1$) and ($\hat{\rm C}_2$) by $\mathcal{S}^p$ and set $\mathcal{S}^p_{\rm per}:=\{f\in\mathcal{S}^p:f\hbox{ is $1$-periodic}\}$. When $0<\eta\leq a\in L^\infty(\RR^N)\cap C(\RR^N)$ we have $H\in\mathcal{S}^p$ since $H$ satisfies (C$_1$) with $\omega(t):={1\over\eta}\sup\{|a(x_1)-a(x_2)|:|x_1-x_2|\leq t\}$. If moreover $a$ is $1$-periodic then $H\in\mathcal{S}^p_{\rm per}$.

The following theorem, whose proof is given below, is a slight improvement of \cite[Proposition 1.8]{oah-jpm08a} (see also \cite[th\'eor\`eme 2.20]{oah-jpm09}).

\begin{theorem}\label{SingularGrowthConditionTheorem}
Let $f:\RR^N\times\MM^{N\times N}\to[0,+\infty]$ be a normal integrand. If $f$ satisfies {\rm($\hat{\rm C}_2$)} then $\Z f$ is of $p$-polynomial growth, i.e., 
\[
\Z f(x,\xi)\leq c(1+|\xi|^p)\hbox{ for all }(x,\xi)\in\RR^N\times\MM^{N\times N}\hbox{ and some } c>0.
\]
\end{theorem}

 By Theorem \ref{SingularGrowthConditionTheorem} we have $\mathcal{S}^p\subset\mathcal{A}^p$ and so $\mathcal{S}^p_{\rm per}\subset\mathcal{A}^p_{\rm per}$. Hence, as a direct consequence of Theorems \ref{SingularGrowthConditionTheorem} and \ref{HomogenizationTheorem} we have
\begin{corollary}\label{corollary1}
If $W\in\mathcal{S}^p_{\rm per}$ then $I_{\rm hom}=\Gamma\hbox{\rm -}\lim_{\eps\to 0}I_\eps$ with $I_{\rm hom}$ defined by \eqref{IntegralFunctHom} and $W_{\rm hom}:\MM^{N\times N}\to[0,+\infty]$ given by \eqref{Homogenized-Density}.
\end{corollary} 

\begin{remark}
From Theorems \ref{SingularGrowthConditionTheorem} and \ref{GeneralizedRelaxationTheorem} we obtain
\begin{corollary}\label{corollary2}
If $f\in\mathcal{S}^p$ then \eqref{IntRep1} and \eqref{IntRep2} hold with $\overline{f}=\Z f=\mathcal{Q} f$.
\end{corollary}
Corollary \ref{corollary2} slightly improves  \cite[Corollaire 2.22]{oah-jpm09} (see also \cite[Theorem 1.3]{oah-jpm08a}).
\end{remark}

To prove Theorem \ref{SingularGrowthConditionTheorem} we need the following two lemmas. The first is a special case of a theorem due to  Dacorogna and Ribeiro (see \cite[Theorem 1.3]{daco-rib04}, see also \cite[Theorem 10.29 p. 462]{dacorogna08})  and the second is a special case of a theorem due to Ben Belgacem (see \cite{benbelgacem96},  see also \cite[Th\'eor\`eme 3.25]{oah-jpm09} for a proof). 

\begin{lemma}\label{lemma1}
Given $t_1<t_2$ and $\xi\in\MM^{N\times N}$ with $t_1<\det\xi<t_2$ there exists $\varphi\in W^{1,\infty}_0(Y;\RR^N)$ such that $\det(\xi+\nabla\varphi(y))\in\{t_1,t_2\}$ for a.e. $y\in Y$.
\end{lemma}

\begin{lemma}\label{lemma2}
Let $f:\RR^N\times\MM^{N\times N}\to[0,+\infty]$ be a normal integrand. If $f$ satisfies {\rm($\hat{\rm C}_2$)} then $\mathcal{R} f$ is of $p$-polynomial growth, where for every $x\in\RR^N$, $\mathcal{R}f(x,\cdot)$ denotes the rank-one convex envelope of $f(x,\cdot)$, i.e., the greatest rank-one convex function which less than or equal to $f(x,\cdot)$.
\end{lemma}

\begin{proof}[Proof of Theorem {\rm\ref{SingularGrowthConditionTheorem}}]
Fix any $x\in\RR^N$ and any $\xi\in\MM^{N\times N}$. Clearly, if $|\det\xi|\geq\alpha$ then $\Z f(x,\xi)<+\infty$. On the other hand, if $|\det\xi|<\alpha$ then, by Lemma \ref{lemma1}, there exists $\varphi\in W^{1,\infty}_0(Y;\RR^N)$ such that $|\det(\xi+\nabla\varphi(y)|=\alpha$ for a.e. $y\in Y$, and using  ($\hat{\rm C}_2$) we see that
\[
\Z f(x,\xi)\leq\int_Y f(x,\xi+\nabla\varphi(y))dy\leq2^p\beta\left(1+|\xi|^p+\|\nabla\varphi\|^p_{L^p(Y;\RR^N)}\right)<+\infty.
\]
Thus $\Z f(x,\xi)<+\infty$ for all $\xi\in\MM^{N\times N}$, i.e., $\Z f(x,\cdot)$ is finite. From Proposition \ref{FonsecaProperties}(b) we deduce that $\Z f(x,\cdot)$ is rank-one convex. Hence $\Z f(x,\cdot)\leq \mathcal{R} f(x,\cdot)$ for all $x\in\RR^N$, i.e., $\Z f\leq\mathcal{R} f$, and the theorem follows from Lemma \ref{lemma2}. 
\end{proof}


\bibliographystyle{acm}

\end{document}